\documentclass[12pt]{article}
\usepackage{amsthm}
\usepackage{amsfonts}
\usepackage{epsfig}
\usepackage{txfonts}
\usepackage{float}

\newtheorem{theorem}{Theorem}

\newtheorem{lemma}[theorem]{Lemma}

\floatstyle{boxed}
\newfloat{algorithm}{hbtp}{lop}
\floatname{algorithm}{Algorithm}
\newcounter{claims}
\newenvironment{claims}{\refstepcounter{claims}\par\medskip\noindent%
 {{\bf (\theclaims)}~~}}{\par\medskip}

\newcommand{\dom}{\mbox{dom}}
\newcommand{\wcol}{\mbox{wcol}}
\newcommand{\col}{\mbox{col}}
\newcommand{\adm}{\mbox{adm}}
\newcommand{\sd}{\mbox{sd}}
\newcommand{\GG}{\cal G}

\title{Constant-factor approximation of domination number in sparse graphs}

\author{Zden\v{e}k Dvo\v{r}\'ak\thanks{Charles University, Prague, Czech Republic.
E-mail: {\tt rakdver@kam.mff.cuni.cz}.
Supported by Institute for Theoretical Computer Science (ITI), project 1M0021620808 of Ministry of Education of Czech Republic.
The work leading to this invention has received funding from the European
Research Council under the European Union's Seventh Framework Programme
(FP7/2007-2013)/ERC grant agreement no. 259385.}}
\date{}

\begin{document}
\maketitle

\begin{abstract}
The \emph{$k$-domination number} of a graph is the minimum size of a set $X$ such that every vertex of $G$ is in distance at most $k$ from $X$.
We give a linear time constant-factor approximation algorithm for $k$-domination number in classes of graphs
with bounded expansion, which include e.g. proper minor-closed graph classes, classes closed on topological minors or
classes of graphs that can be drawn on a fixed surface with bounded number of crossings on each edge.

The algorithm is based on the following approximate min-max characterization.
A subset $A$ of vertices of a graph $G$ is \emph{$d$-independent} if the distance between each
pair of vertices in $A$ is greater than $d$.  Note that the size of the largest $2k$-independent set is a lower bound
for the $k$-domination number.  We show that every graph from a fixed class with bounded expansion
contains a $2k$-independent set $A$ and a $k$-dominating set $D$ such that $|D|=O(|A|)$,
and these sets can be found in linear time.  For domination number ($k=1$) the assumptions can be relaxed,
and the result holds for all graph classes with arrangeability bounded by a constant.
\end{abstract}

\section{Introduction}

For an undirected graph $G$, a set $D\subseteq V(G)$ is \emph{dominating} if every vertex $v\in V(G)\setminus D$ has a neighbor in $D$.
Determining the minimal size $\dom(G)$ of a dominating set in $G$ is NP-complete in general (Karp~\cite{Karp}).  Moreover, even approximating it within factor
better than $O(\log |V(G)|)$ is NP-complete (Raz and Safra~\cite{raz1997sub}).  On the other hand, the problem becomes more manageable when
restricted to some special classes of sparse graphs.  For example, there exists a PTAS for dominating set in planar graphs (Baker~\cite{baker1994approximation}).

In this paper, we follow the approach of B{\"o}hme and Mohar~\cite{bohme2003domination}.  A subset $A$ of vertices of a graph $G$ is \emph{$d$-independent} if the distance between each
pair of vertices in $A$ is greater than $d$.  Denote by $\alpha_d(G)$ the maximum size of a $d$-independent set in $G$. Clearly, every vertex of $G$ has at most one neighbor in a $2$-independent
set; hence, we have $\dom(G)\ge \alpha_2(G)$.  In general, it is not possible to give an upper bound on $\dom(G)$ in the terms of $\alpha_2(G)$; see Section~\ref{sec-lower} for examples of graphs
with $\alpha_2(G)=2$, but unbounded domination number.  However, B{\"o}hme and Mohar~\cite{bohme2003domination} proved that for graphs in any proper minor-closed class, $\dom(G)$ is bounded
by a linear function of $\alpha_2(G)$.

\begin{theorem}[B{\"o}hme and Mohar~\cite{bohme2003domination}, Corollary 1.2]\label{thm-bojan}
If $G$ does not contain $K_{q,r}$ as a minor, then $\dom(G)\le (4r+(q-1)(r+1)) \alpha_2(G) - 3r$.
\end{theorem}

The proof of the theorem is constructive, giving a polynomial-time algorithm that finds a dominating set $D$ and a $2$-independent set $A$
such that $|D|\le (4r+(q-1)(r+1))|A|-3r$.  Since $|A|\le \dom(G)$, this approximates $\dom(G)$ within the constant factor $4r+(q-1)(r+1)$.

We generalize Theorem~\ref{thm-bojan} by relaxing the assumption on the considered class of graphs.  First, let us introduce several closely related graph parameters.
Let $v_1$, $v_2$, \ldots, $v_n$ be an ordering of the vertices of a graph $G$.  A vertex $v_a$ is \emph{weakly $k$-accessible} from $v_b$
if $a<b$ and there exists a path $v_a=v_{i_0}, v_{i_1}, \ldots, v_{i_{\ell}}=v_b$ of length $\ell\le k$ in $G$ such that
$a\le i_j$ for $0\le j\le \ell$.  We say that $v_a$ is \emph{$k$-accessible} from $v_b$ if additionally $b\le i_j$ for $1\le j\le \ell$.
For a fixed ordering of $V(G)$, let $Q_k(v)$ denote the set of vertices that are weakly $k$-accessible from $v$, $R_k(v)$ the set of vertices that are weakly $k$-accessible from $v$
and let $q_k(v)=|Q_k(v)|$ and $r_k(v)=|R_k(v)|$.
The \emph{$k$-backconnectivity} $b_k(v)$ of $v$ with respect to the fixed ordering of $V(G)$ is the maximum number of paths from $v$ of length
at most $k$ that intersect only in $v$, such that all endvertices of these paths distinct from $v$ appear before $v$ in the ordering
(clearly, we can assume that the internal vertices of the paths appear after $v$ in the ordering).  Note that $b_k(v)\le r_k(v)\le q_k(v)$.
The \emph{weak $k$-coloring number}, \emph{$k$-coloring number} and \emph{$k$-admissibility} of the ordering is the maximum of $1+q_k(v)$,
$1+r_k(v)$ and $b_k(v)$, respectively, over $v\in V(G)$.  The weak $k$-coloring number $\wcol_k(G)$ of $G$ is the minimum of the weak $k$-coloring numbers
over all orderings of $V(G)$, and the $k$-coloring number $\col_k(G)$ and $k$-admissibility $\adm_k(G)$ of $G$ are defined analogically.

Obviously, $\adm_k(G)<\col_k(G)\le \wcol_k(G)$.  Conversely, it is easy to see that $\wcol_k(G)\le \col^k_k(G)$ (Kierstead and Yang~\cite{kierstead2003orderings})
and that $\col_k(G)\le \adm_k^k(G)+1$ (Lemma~\ref{lemma-admrel} in Section~\ref{sec-order}).
Let us remark that $\wcol_1(G)-1=\col_1(G)-1=\adm_1(G)$ is equal to the degeneracy of $G$, and that $\col_2(G)-1$ and $\adm_2(G)$
are known as the \emph{arrangeability} and \emph{admissibility} of $G$, respectively, in the literature (see e.g. \cite{arr1}, \cite{arr2} or \cite{arr3}).
For the domination number, our main result can be stated as follows.

\begin{theorem}\label{thm-mainspec}
If $G$ satisfies $\wcol_2(G)\le c$, then $\dom(G)\le c^2\alpha_2(G)$.
\end{theorem}

The proof gives a linear-time algorithm to find the corresponding dominating and $2$-independent sets, assuming that the ordering of the vertices
of $G$ with weak $2$-coloring number at most $c$ is given.  We discuss the algorithmic and complexity aspects of obtaining such an ordering in Section~\ref{sec-order}.
To relate Theorem~\ref{thm-mainspec} to Theorem~\ref{thm-bojan}, we use the following characterization.  For an integer $t\ge 0$ and a graph $G$, let $\sd_t(G)$
denote the graph obtained from $G$ by subdividing each edge exactly $t$ times.

\begin{theorem}[Dvo\v{r}\'ak~\cite{dvorak2008forbidden}, Theorem 9]\label{thm-char}
Let $G$ be a graph and $d$ an integer.
If $\delta(H)<d$ for every $H$ such that $H\subseteq G$ or $\sd_1(H)\subseteq G$, then $\col_2(G)\le 4d^2(4d+5)+1$.
\end{theorem}

Conversely, let us note that if $\delta(H)=d$, then $\col_2(\sd_1(H))\ge\adm_2(\sd_1(H))\ge d$, which is easy to see by considering the last vertex of degree at least $d$ in
the optimal ordering for $2$-admissibility.  Consider now a proper minor-closed graph class $\GG$.  There exists a constant $c$ such that all graphs in $\GG$ have minimum degree
less than $c$ (Kostochka~\cite{kostomindeg}).  Now, if $\sd_1(H)\subseteq G$ for a graph $G\in \GG$, then $H$ is a minor of $G$ and belongs to $\GG$ as well,
and thus $\delta(H)\le c$.  Theorem~\ref{thm-char} thus implies that $\col_2(G)=O(c^3)$ and we can apply Theorem~\ref{thm-mainspec} for $G$.  Therefore, we indeed generalize
Theorem~\ref{thm-bojan}, although the multiplicative constant in our result may be greater.  More generally, the same argument shows that Theorem~\ref{thm-mainspec} applies to all graph classes closed on topological subgraphs.

B{\"o}hme and Mohar~\cite{bohme2003domination} in fact proved a more general result concerning distance domination.  A set $D\subseteq V(G)$ is \emph{$k$-dominating} if
the distance from any vertex of $G$ to $D$ is at most $k$; thus, $1$-dominating sets are precisely dominating sets.  Let $\dom_k(G)$ denote the size of the smallest
$k$-dominating set in $G$.  Clearly, $\dom_k(G)\ge \alpha_{2k}(G)$.  Theorem 1.1 of \cite{bohme2003domination} shows that in any proper minor-closed class of graphs,
$\dom_k(G)=O(\alpha_m(G))$, for any $m<\frac{5}{4}(k+1)$.  We strengthen this result by considering less restricted classes of graphs as well as increasing
$m$ to the natural bound:

\begin{theorem}\label{thm-main}
If $1\le m\le 2k+1$ and $G$ satisfies $\wcol_m(G)\le c$, then $\dom_k(G)\le c^2\alpha_m(G)$.  Furthermore, if
an ordering of $V(G)$ such that $q_m(v)< c$ for every $v\in V(G)$ is given, then a $k$-dominating set $D$ and
an $m$-independent set $A$ such that $|D|\le c^2|A|$ can be found in $O(c^2\max(k,m)|V(G)|)$ time.
\end{theorem}

The bound $2k+1$ on $m$ instead of $2k$ may seem surprising at first.  It is caused by the following parity reason: suppose that $T$ is a $2k$-independent set
and $v$ a vertex such that for every pair of vertices $x,y\in T$, the shortest path between $x$ and $y$ passes through $v$.  Since $T$ is $2k$-independent,
at most one vertex of $T$ is in distance at most $k$ from $v$.  Therefore, $T$ contains a $(2k+1)$-independent subset of size at least $|T|-1$.

For which graph classes can Theorem~\ref{thm-main} be applied for every $k\ge 0$? I.e., for what graph classes does there exist
a function $f$ such that $\wcol_m(G)\le f(m)$ for every graph $G$ in the class?  By Zhu~\cite{zhu2009colouring}, these are precisely
the graph classes with \emph{bounded expansion} (see Ne\v{s}et\v{r}il and Ossona de Mendez~\cite{nessurvey} for various equivalent definitions
and properties of such graph classes).  Let us note that most classes of ``structurally sparse'' graphs have bounded expansion,
including proper graph classes closed on topological minors and graphs that can be drawn in a fixed surface with bounded number of crossings
on each edge.

\section{Proof of the main result}

Theorem~\ref{thm-mainspec} is a special case of Theorem~\ref{thm-main} with $k=1$ and $m=2$, thus it suffices to prove the latter.
We defer the discussion of the algorithmic aspects to Section~\ref{sec-order}, and prove here just the existence of the sets $D$ and $A$
with the required properties.

\begin{proof}[Proof of Theorem~\ref{thm-main}]
Let $v_1$, \ldots, $v_n$ be an ordering of vertices of $G$ such that $q_m(v)\le c-1$ for every $v\in V(G)$.  We construct
sets $D$ and $A'$ using Algorithm~\ref{alg-main}.  Clearly, $D$ is a $k$-dominating set in $G$ and $|D|\le c|A'|$.
\begin{algorithm}
\begin{itemize}
\item initialize $D\coloneqq\emptyset$, $A'\coloneqq\emptyset$ and $R\coloneqq V(G)$
\item while $R$ is nonempty, repeat:
\begin{itemize}
\item let $v$ be the first vertex of $R$ in the ordering
\item set $A'\coloneqq A'\cup\{v\}$
\item set $D\coloneqq D\cup \{v\}\cup Q_m(v)$
\item remove from $R$ all vertices whose distance from $\{v\}\cup Q_m(v)$ is at most~$k$
\end{itemize}
\end{itemize}
\caption{Finding the dominating set}\label{alg-main}
\end{algorithm}

For each $w\in A'$, let $T_w$ be the set of vertices $a_t\in A'$ such that $w\in \{a_t\}\cup Q_k(a_t)$.
Let $H$ be the graph with vertex set $A'$ such that
$uv\in E(H)$ iff the distance between $u$ and $v$ in $G$ is at most $m$.  Let $a_1$, $a_2$, \ldots, $a_s$ be the vertices of $H$ in
the order consistent with the ordering of $V(G)$. 

Consider vertices $a_i, a_j\in V(H)$ such that $j<i$ and
$G$ contains a path $P$ of length at most $m$ between $a_i$ and $a_j$.  Let $z$ be the first vertex of $P$ according to the ordering
of $V(G)$.  Observe that $z\in Q_m(a_i)\cap (\{a_j\}\cup Q_m(a_j))$.  By the construction of $A'$, the distance of $a_i$ from
$\{a_j\}\cup Q_m(a_j)$ is at least $k+1$, and thus
the length of the subpath of $P$ between $a_j$ and $z$ is at most $m-k-1\le k$.  Therefore, we have $a_j\in T_z$.  It follows that if $a_j$ is
$1$-accessible from $a_i$ in $H$, then $a_j\in \bigcup_{w\in Q_m(a_i)} T_w$.  On the other hand, we have $|T_w|\le 1$ for every $w\in A'$,
since if $x\in T_w$, then all the vertices whose distance from $w$ is at most $k$ were removed from $R$ when we added $x$ to $A'$.  Therefore, the number of
vertices of $H$ that are $1$-accessible from $a_i$ is at most $q_m(a_i)\le c-1$.

We conclude that $\col_1(H)\le c$.  Since $\col_1(H)\ge \chi(H)$, the graph $H$ has an independent set $A$ of size at least $|A'|/c$.
By the definition of $H$, the set $A$ is $m$-independent in $G$, and we have $|D|\le c|A'|\le c^2|A|$ as required.
\end{proof}

\section{Algorithmic aspects}\label{sec-order}

Let $G$ be a graph on $n$ vertices such that $\wcol_m(G)\le c$.  First, assume that we are given an ordering of
$V(G)$ such that $q_m(v)< c$ for every $v\in V(G)$.  Since $m\ge 1$, this implies that $G$ is $c$-degenerate, and thus it has at most $cn$ edges.

For each $i\le m$ and $v\in V(G)$, we determine the set $Q_i(v)$ (whose size is bounded by $c$) using the following algorithm:
For $i=1$, $Q_1(v)$ is the set of neighbors of $v$
that appear before it in the ordering, which can be determined by enumerating all the edges incident with $v$.  For $i>1$, $Q_i(v)$ is the subset
of $Q_1(v)\cup \bigcup_{uv\in E(G)} Q_{i-1}(u)$ consisting of the vertices before $v$ in the ordering.  Note that $Q_i(v)$ can be determined
in $O(c(\deg(v)+1))$, assuming that $Q_{i-1}$ was already computed before.  Therefore, each $Q_i$ can be computed for all vertices of $G$
in $O(c^2n)$, and in total we spend time $O(c^2mn)$ to determine $Q_m(v)$ for every vertex of $G$.

With this information, we can implement Algorithm~\ref{alg-main} in time $O(c(k+1)n)$.
The only nontrivial part is the removal of the vertices from $R$.  For each vertex $v$ of $V(G)$ we maintain the value $p(v)=\min(k + 1, d(v))$,
where $d(v)$ is the distance of $v$ from $D$.  In each step, we have $v\in R$ iff $p(v)=k+1$ and $v\in D$ iff $p(v)=0$.  When a vertex $v$ is added to $D$,
we decrease $p(v)$ to $0$.  For each vertex $w$, whenever the value of $p(w)$ decreases, we recursively propagate this change to the neighbors of $w$:
if $uw\in E(G)$ and $p(u)>p(w)+1$, then we decrease $p(u)$ to $p(w)+1$.  Clearly, the value of $p(w)$ decreases at most $(k+1)$ times during the
run of the algorithm, and we spend time $O((k+1)\deg(v))$ by updating it and propagating the decrease to the neighbors.  Therefore, the total time
for maintaining the set $R$ is bounded by $O(c(k+1)n)$.

For the final part of the algorithm, we need to determine the edges of $H$.  First we compute the set $T_w$ for each vertex $w\in V(G)$:
we initialize these sets to $\emptyset$, and then for each $a\in A'$, we add $a$ to $T_w$ for each $w\in \{a\}\cup Q_k(a)$.  A supergraph $H'$ of $H$
with $\col_1(H')\le c$ is then obtained by joining each $a\in A'$ with all the elements of $\bigcup_{w\in Q_m(a)} T_w$ that precede $a$ in the ordering.
We find a proper coloring of $H'$ by at most $c$ colors using the standard greedy algorithm, and choose $A$ as the largest color class in this coloring.
The time for this phase is $O(cn)$.

Therefore, the total complexity of the algorithm is $O(c^2\max(m,k)n)$.
The space complexity is bounded by the space needed to represent $Q_k$ and $Q_m$, and thus it is $O(cn)$.

Let us now turn our attention to the problem of finding a suitable ordering of vertices.  We were not able to find a polynomial-time algorithm to
determine $\wcol_m(G)$ for $m\ge 2$, and we conjecture that the problem is NP-complete.  However, determining $\adm_m(G)$ appears to
be easier, and the corresponding ordering has also bounded weak $m$-coloring number.

\begin{lemma}\label{lemma-admrel}
Let $G$ be a graph and $v_1$, $v_2$, \ldots, $v_n$ an ordering of its vertices with $m$-admissibility at most $c$.
Then the $m$-coloring number of the ordering is at most $c(c-1)^{m-1}+1$.
\end{lemma}
\begin{proof}
Consider a vertex $v\in V(G)$.  There exists a tree $T\subseteq G$ rooted in $v$ such that $R_m(v)$ is the set of leaves of $T$,
every path of $T$ starting in $v$ has length at most $m$ and all non-leaf vertices of $T$ distinct from $v$ appear after $v$ in the ordering.
Observe that every non-leaf vertex $u\in V(T)$ has degree at most $c$ in $T$, as otherwise there would exist at least $c+1$ paths in $G$ from $u$ of length at most $m$
intersecting only in $u$ and ending before $u$ (in $\{v\}\cup R_m(v)$), contradicting the assumption that the $m$-backconnectivity of $u$ is at most $c$.
We conclude that $T$ has at most $c(c-1)^{m-1}$ leaves, and thus $r_m(v)\le c(c-1)^{m-1}$.  Since this holds for every vertex $v\in V(G)$,
the claim of the lemma follows.
\end{proof}

Together with the observation of Kierstead and Yang~\cite{kierstead2003orderings}, this implies that the weak $m$-coloring number of the ordering is at most $(c(c-1)^{m-1}+1)^m$.
For a set $S\subseteq V(G)$ and $v\in S$, let $b_m(S, v)$ be the maximum number of paths from $v_i$ of length at most $m$ intersecting only in $v_i$ whose internal vertices belong to $V(G)\setminus S$ and endvertices belong to $S$.
The ordering of $V(G)$ with the smallest $m$-admissibility can be found using Algorithm~\ref{alg-adm}.
\begin{algorithm}
\begin{itemize}
\item initialize $S\coloneqq V(G)$
\item for $i=n, n-1, \ldots, 1$:
\begin{itemize}
\item choose $v_i\in S$ minimizing $p_i=b_m(S, v_i)$
\item set $S\coloneqq S\setminus\{v_i\}$
\end{itemize}
\end{itemize}
\caption{Determining $m$-admissibility}\label{alg-adm}.
\end{algorithm}
Clearly, the resulting ordering has $m$-admissibility $\max(p_1,\ldots, p_n)$, and it is easy to see that this is equal to $\adm_m(G)$: Suppose that there exists an ordering $X$ of $V(G)$ with
$m$-admissibility at most $p$, and consider an arbitrary set $S\subseteq V(G)$.  Let $v$ be the last vertex of $S$ according to the ordering $X$.  Then $b_m(S, v)\le p$, since all vertices
of $S$ are before $v$ in the ordering $X$ and the $m$-backconnectivity of $v$ in $X$ is at most $p$.  Therefore, we have $p_i\le p$ for $1\le i\le n$ in the algorithm.

A bit problematic step in the algorithm is finding the vertex $v\in S$ minimizing $b_m(S, v)$, since for $m\ge 5$, determining $b_m(S, v)$ is NP-complete in general (Itai, Perl and Shiloach~\cite{ips}).
Nevertheless, $b_m(S, v)$ can be approximated within factor of $m$, by repeatedly taking any path from $v$ to $S$ of length at most $m$ with all internal vertices in $V(G)\setminus S$ and removing its vertices distinct from
$v$ (each such removal can interrupt at most $m$ paths in the optimal solution, hence we will be able to pick at least $b_m(S, v)/m$ paths this way).  A straightforward implementation gives an $O(mn^3)$
algorithm to approximate $\adm_m(G)$ within factor of $m$.

When the $m$-admissibility of $G$ is bounded by a constant $p$, we can obtain a polynomial-time algorithm to determine $\adm_m(G)$ exactly.  To test whether $b_m(S, v)\le p$,
we simply enumerate all sets of at most $p+1$ paths of length at most $m$ starting in $v$.  A straightforward implementation gives an algorithm with time complexity
$O(n^{mp+m+2})$.  This time complexity can be improved significantly if the considered class of graphs $\GG$ has bounded expansion.  Dvo\v{r}\'ak et al.~\cite{dvorak2010deciding}
described a data structure to represent a graph in such a class and answer first-order queries for it in a constant time.
In particular, suppose that $\varphi(x)$ is a first-order formula with one free variable $x$ using a binary predicate $e$ and a unary predicate $s$.
This data structure can be used to represent a graph in $\GG$ and a subset $S$ of its vertices, so that
\begin{itemize}
\item the data structure can be initialized in linear time,
\item we can add a vertex to $S$ or remove it from $S$ in constant time, and
\item we can find in constant time a vertex $v\in V(G)$ such that $\varphi(v)$ holds, with $e$ interpreted as the adjacency in $G$
and $s$ as the membership in $S$, or decide that no such vertex exists.
\end{itemize}

For the purpose of the algorithm for $m$-admissibility, to test whether $b_m(S, v_i)\le p$, we apply the data structure for the property
$$\varphi(x)=s(x)\land \lnot\left[(\exists y_1)\ldots (\exists y_{m(p+1)}) \psi(x, y_1,\ldots, y_{m(p+1)})\right],$$
where $\psi$ is the formula describing that the subgraph induced by $\{x, y_1, \ldots, y_{m(p+1)}\}$ contains
$p+1$ paths from $x$ of length at most $m$, intersecting only in $x$, and with endvertices satisfying $s$ and internal vertices not satisfying $s$.

Using this data structure, we repeatedly find $x\in S$ such that $b_m(S, x)\le p$ and remove it from to $S$, thus obtaining an ordering of $V(G)$
with admissibility at most $p$ or determining that $\adm_m(G)>p$ in linear time.
By Zhu~\cite{zhu2009colouring}, for each class $\GG$ with bounded expansion, there exists a function $f$ such that $\adm_m(G)\le f(m)$ for each $G\in \GG$.
Therefore, we can determine the exact value of the $m$-admissibility by applying this test for $p=1,\ldots, f(m)$.

\begin{theorem}
Let $\GG$ be a class of graphs with bounded expansion and $m\ge 1$ an integer.  There exists a linear-time algorithm
that for each $G\in \GG$ determines $\adm_m(G)$ and outputs the corresponding ordering of $V(G)$.
\end{theorem}

By combining this algorithm with Theorem~\ref{thm-main}, we obtain the following result.

\begin{theorem}\label{thm-algbexp}
Let $\GG$ be a class of graphs with bounded expansion and $k\ge 1$ a constant.  There exists an algorithm that for each $G\in \GG$ returns
a $k$-dominating set $D$ and a $(2k+1)$-independent set $A$ such that $|D|=O(|A|)$.  The algorithm runs in time $O(|V(G)|)$.
\end{theorem}

\section{Lower bound}\label{sec-lower}

Let us now explore the limits for the possible extensions of Theorem~\ref{thm-main}.
For $n\ge 3$, let $G'_n=\sd_{2k-1}(K_n)$, let $X$ be the set of the middle vertices of the paths corresponding to the
edges of $K_n$ in $G'_n$ and let $Y$ be the set of vertices of $G'_n$ of degree $n-1$.  Let $G_n$ be the graph obtained from $G'_n$ by
adding a new vertex $v$ adjacent to all the vertices of $X$.

The distance between any two vertices of $V(G_n)\setminus Y$ is at most $2k$, since all these vertices are in distance at most $k$ from $v$.
Furthermore, the distance between any two vertices of $Y$ is most $2k$, since they are joined by a path of length $2k$ corresponding
to an edge of $K_n$.  Therefore, $\alpha_{2k}(G_n)\le 2$.  On the other hand, for any $w\in V(G_n)\setminus X$, there is at most one vertex
of $Y$ whose distance from $w$ is at most $k$, and each vertex of $X$ has distance at most $k$ from exactly two vertices of $Y$.
Therefore, $\dom_k(G_n)\ge n/2$.  Therefore, $k$-domination number cannot be bounded by a function of $2k$-independence number
on any class of graphs that contains $\{G_i : i\ge 3\}$.

Let us consider the following ordering of the vertices of $G_n$: the first vertex is $v$, followed by $Y$ in an arbitrary order, followed
by the rest of vertices of $G_n$ in an arbitrary order.  Since the distance between any two vertices of $Y$ is $2k$, we have
$q_{2k-1}(w)\le 1$ for $w\in Y$, and similarly $q_{2k-1}(w)\le 2k+1$ for $w\in V(G_n)\setminus Y$.  Therefore, $\wcol_{2k-1}(G_n)\le 2k+2$ for every $n\ge 3$.
It follows that at least in the case that $m=2k$, it is not sufficient to restrict $\wcol_{2k-1}(G)$ in Theorem~\ref{thm-main}.

Another possible extension, bounding $\dom_k(G)$ by a function of $\alpha_{2k+2}(G)$, is impossible even for trees~\cite{bohme2003domination},
as the graph $\sd_k(K_{1,n})$ demonstrates.

\section*{Acknowledgements}

I would like to thank Bojan Mohar for bringing the problem to my attention and for useful discussions regarding it.

\bibliographystyle{siam}
\bibliography{apxdomin}

\begin{thebibliography}{10}

\bibitem{baker1994approximation}
{\sc B.~Baker}, {\em Approximation algorithms for np-complete problems on
  planar graphs}, Journal of the ACM (JACM), 41 (1994), pp.~153--180.

\bibitem{bohme2003domination}
{\sc T.~B{\"o}hme and B.~Mohar}, {\em Domination, packing and excluded minors},
  Electronic Journal of Combinatorics, 10 (2003), p.~N9.

\bibitem{arr2}
{\sc G.~Chen and R.~Schelp}, {\em Graphs with linearly bounded {R}amsey
  numbers}, J. Comb. Theory Ser. B, 57 (1993), pp.~138--149.

\bibitem{dvorak2008forbidden}
{\sc Z.~Dvo\v{r}\'ak}, {\em On forbidden subdivision characterizations of graph
  classes}, European Journal of Combinatorics, 29 (2008), pp.~1321--1332.

\bibitem{dvorak2010deciding}
{\sc Z.~Dvo\v{r}\'ak, D.~Kr\'al', and R.~Thomas}, {\em Deciding first-order
  properties for sparse graphs}, in 2010 IEEE 51st Annual Symposium on
  Foundations of Computer Science, 2010, pp.~133--142.

\bibitem{ips}
{\sc A.~Itai, Y.~Perl, and Y.~Shiloach}, {\em The complexity of finding maximum
  disjoint paths with length constraints}, Networks, 12 (1982), pp.~277--286.

\bibitem{Karp}
{\sc R.~Karp}, {\em Reducibility among combinatorial problems}, in Complexity
  of Computer Computations, R.~Miller and J.~Thatcher, eds., Plenum Press,
  1972, pp.~85--103.

\bibitem{arr3}
{\sc H.~Kierstead and W.~Trotter}, {\em Planar graph coloring with an
  uncooperative partner}, J. Graph Theory, 18 (1994), pp.~569--584.

\bibitem{kierstead2003orderings}
{\sc H.~Kierstead and D.~Yang}, {\em Orderings on graphs and game coloring
  number}, Order, 20 (2003), pp.~255--264.

\bibitem{kostomindeg}
{\sc A.~Kostochka}, {\em Lower bound of the {H}adwiger number of graphs by
  their average degree}, Combinatorica, 4 (1984), pp.~307--316.

\bibitem{nessurvey}
{\sc J.~Ne{\v{s}}et\v{r}il and P.~O. de~Mendez}, {\em Structural properties of
  sparse graphs}, Bolyai Society Mathematical Studies, 19 (2008), pp.~369--426.

\bibitem{raz1997sub}
{\sc R.~Raz and S.~Safra}, {\em A sub-constant error-probability low-degree
  test, and a sub-constant error-probability pcp characterization of np}, in
  Proceedings of the twenty-ninth annual ACM symposium on Theory of computing,
  ACM, 1997, pp.~475--484.

\bibitem{arr1}
{\sc V.~Rodl and R.~Thomas}, {\em Arrangeability and clique subdivisions},
  Algorithms Combin., 14 (1997), pp.~236--239.

\bibitem{zhu2009colouring}
{\sc X.~Zhu}, {\em Colouring graphs with bounded generalized colouring number},
  Discrete Mathematics, 309 (2009), pp.~5562--5568.

\end{thebibliography}

\end{document}